\pdfoutput=1
\RequirePackage{ifpdf}
\ifpdf 
\documentclass[pdftex]{sigma}
\else
\documentclass{sigma}
\fi
\input xy
\xyoption{all}

\numberwithin{equation}{section}

\newtheorem{Theorem}{Theorem}[section]

\newtheorem{Lemma}[Theorem]{Lemma}
\newtheorem{Proposition}[Theorem]{Proposition}
 { \theoremstyle{definition}

\newtheorem{Remark}[Theorem]{Remark} }

\DeclareMathOperator{\rk}{rk} 
\DeclareMathOperator{\Pic}{Pic} 
\DeclareMathOperator{\pdeg}{pardeg} 
\DeclareMathOperator{\pmu}{par\mu} 

\DeclareMathOperator{\End}{End} 
\DeclareMathOperator{\tr}{Tr} 

\DeclareMathOperator{\GL}{GL}

\DeclareMathOperator{\GPH}{GPH}
\DeclareMathOperator{\Hitchin}{Hitchin}

\newcommand{\cC}{\mathcal{C}}

\newcommand{\cF}{\mathcal{F}}
\newcommand{\cG}{\mathcal{G}} 

\newcommand{\cO}{\mathcal{O}} 
\newcommand{\cM}{\mathcal{M}} 
\newcommand{\cS}{\mathcal{S}} 

\newcommand{\cH}{\mathcal{H}}

\newcommand{\CC}{\mathbb{C}} 

\newcommand{\PP}{\mathbb{P}} 
\newcommand{\QQ}{\mathbb{Q}} 

\newcommand{\ZZ}{\mathbb{Z}} 

\begin{document}

\allowdisplaybreaks

\newcommand{\arXivNumber}{1809.06021}

\renewcommand{\thefootnote}{}

\renewcommand{\PaperNumber}{024}

\FirstPageHeading

\ShortArticleName{On Higgs Bundles on Nodal Curves}

\ArticleName{On Higgs Bundles on Nodal Curves\footnote{This paper is a~contribution to the Special Issue on Geometry and Physics of Hitchin Systems. The full collection is available at \href{https://www.emis.de/journals/SIGMA/hitchin-systems.html}{https://www.emis.de/journals/SIGMA/hitchin-systems.html}}}

\Author{Marina LOGARES}

\AuthorNameForHeading{M.~Logares}

\Address{School of Computing Electronics and Mathematics, University of Plymouth,\\ Drake Circus, PL4 8AA, UK}
\Email{\href{mailto:marina.logares@plymouth.ac.uk}{marina.logares@plymouth.ac.uk}}
\URLaddress{\url{https://www.plymouth.ac.uk/staff/marina-logares}}

\ArticleDates{Received October 16, 2018, in final form March 14, 2019; Published online March 28, 2019}

\Abstract{This is a review article on some applications of generalised parabolic structures to the study of torsion free sheaves and $L$-twisted Hitchin pairs on nodal curves. In particular, we survey on the relation between representations of the fundamental group of a nodal curve and the moduli spaces of generalised parabolic bundles and generalised parabolic $L$-twisted Hitchin pairs on its normalisation as well as on an analogue of the Hitchin map for generalised parabolic $L$-twisted Hitchin pairs.}

\Keywords{Higgs bundles; nodal curves; generalised parabolic structures}

\Classification{14H60; 14D20}

\renewcommand{\thefootnote}{\arabic{footnote}}
\setcounter{footnote}{0}

\section{Introduction}
The moduli spaces of Higgs bundles on a Riemann surface are rich geometric objects which have been studied extensively in the last decades. They were introduced in~\cite{hitchin:1987} as a special class of solutions to the self-dual Yang--Mills equations. These are equations for a connection and a principal bundle on the Euclidean 4-space but, after a dimensional reduction process and by considering their conformal invariance property, can be defined on a Riemann surface. The corresponding solutions are called Higgs bundles.

A \emph{Higgs bundle} on a compact Riemann surface $X$, consist of a holomorphic vector bundle on~$X$ together with a section of the endomorphism of the bundle twisted with the canonical bundle of~$X$ which is called the \emph{Higgs field}.

One of the most celebrated results regarding this moduli space is the so called non-abelian Hodge correspondence which extends a result by Narasimhan and Seshadri in \cite{narasimhan-seshadri:1965} regarding an isomorphism between the moduli space of stable vector bundles of rank~$r$ and degree~$0$ on a Riemann surface $X$ and the moduli space of unitary representations of the fundamental group of~$X$. The non-abelian Hodge correspondence was proved by the combined results of Corlette~\cite{corlette:1988}, Donaldson \cite{donaldson:1987}, Hitchin \cite{hitchin:1987} and Simpson \cite{simpson:1992}. Under this correspondence, the moduli space of stable Higgs bundles of rank~$r$ and degree~$0$ is analytically isomorphic to the moduli space of irreducible representations of the fundamental group of $X$ into $\GL(r,\CC)$.

Many generalisations of this correspondence and the objects involved have been considered along the last decades. For instance, we may want to study representations of the fundamental group of~$X$ with a~finite set of points $\{x_{i}\}_{i=1}^{s}$ removed, that is $\rho\colon \pi_{1}\big(X\setminus\{x_{i}\}_{i=1}^{s}\big)\rightarrow \GL(r,\CC)$. The corresponding objects in this case are the so called \emph{parabolic Higgs bundles}, which consist of Higgs bundles on~$X$ together with an extra structure called \emph{parabolic structure} defined over the points $\{x_{i}\}_{i=1}^{s}$. In~\cite{simpson:1990} Simpson showed an equivalence of categories between parabolic Higgs bundles and filtered local systems and proved that the non-abelian Hodge correspondence extends to the non-compact situation. In particular, the moduli space of representations of $\pi_{1}\big(X\setminus\{x_{i}\}_{i=1}^{s}\big)$ into $\GL(r,\CC)$ with prescribed monodromy around the points $\{x_{i}\}_{i=1}^{s}$ is real analytically isomorphic to the moduli space of stable parabolic Higgs bundles of rank~$r$, parabolic degree~$0$ (as defined in~(\ref{def:pdeg})) and prescribed parabolic structure around the punctures. Our aim in the present paper is to review the situation where instead of a punctured Riemann surface we have an irreducible algebraic curve with only nodal singularities. In order to do so we shall need to further explore the notion of parabolic structure on a bundle.

Parabolic structures on bundles were introduced by Seshadri in~\cite{seshadri:1969}. In~\cite{mehta-seshadri:1980} Mehta and Seshadri constructed the moduli space of parabolic bundles and proved that the set of equivalent classes of irreducible representations of the fundamental group of a punctured Riemann surface, can be identified with equivalence classes of stable parabolic bundles of parabolic degree $0$ on the Riemann surface. These parabolic structures were generalised to the so called \emph{generalised parabolic structures} by Bhosle in~\cite{bhosle:1992}. Generalised parabolic line bundles appeared already in the work of Oda and Seshadri \cite{oda-seshadri} when studying a~desingularisation of the compactified Jacobian of an integral curve. In~\cite{bhosle:1992} Bhosle constructed the moduli space of generalised parabolic bundles for any rank and gave conditions for it being non-singular. Moreover, she applied her results to the study of the moduli space of torsion free sheaves on a nodal curve. She also proved that given an integral projective curve $Y$ with only nodes as singularities, and its normalisation, $p\colon X\rightarrow Y$, the moduli space of generalised parabolic bundles of rank one on~$X$ is a non-singular projective variety and it is in fact a desingularisation of the compactified Jacobian of~$Y$.

The study of torsion free sheaves on $Y$ by means of bundles with extra structure on $X$ was already addressed by Seshadri in \cite{seshadri:1982}. Bhosle in \cite{bhosle:1992} proves a very successful correspondence between generalised parabolic bundles on~$X$ and torsion free sheaves on~$Y$. This correspondence preserves the degree and rank, thus the moduli space of generalised parabolic bundles turns out to be an excellent tool for the study of the moduli spaces of torsion free sheaves on~$Y$.

As we already mentioned, in \cite{mehta-seshadri:1980} Mehta and Seshadri proved that the well-known theorem of Narasimhan and Seshadri~\cite{narasimhan-seshadri:1965} extends to a correspondence between bundles with parabolic structures and unitary representations of punctured Riemann surfaces. Hence, a natural question arises: what would be the relation between the representations of the fundamental group of $Y$ and the moduli space of generalised parabolic bundles. Bhosle addresses this question in~\cite{bhosle:1995} and shows that there is no analogue of the Narasimhan--Seshadri theorem in the nodal case.

Bhosle studied non-unitary representations by introducing the moduli space of \emph{generalised parabolic $L$-twisted Hitchin pairs} on~$X$ in~\cite{bhosle:2014}. These are Higgs bundles twisted by a line bundle~$L$ instead of the canonical bundle of~$X$ together with a generalised parabolic structure. She also finds a correspondence with the moduli space of $L$-twisted Hitchin bundles on $Y$. But the correspondence between generalised parabolic $L$-twisted Hitchin pairs on~$X$ and representations of the fundamental group of~$Y$ is not as straightforward as the non-abelian Hodge correspondence for Higgs bundles (and parabolic Higgs bundles) on~$X$.

\looseness=-2 In \cite{bhosle-parameswaran} Bhosle and Parameswaran introduced the notion of strong semistabilty for bundles on singular curves and constructed a group scheme, $\cG_{Y}$, associated to $Y$ such that strongly semistable bundles on $Y$ come from representations $\operatorname{Rep}(\cG_{Y})$, of $\cG_{Y}$ into the general linear group. In~\cite{bhosle-parameswaran-singh} Bhosle, Parameswaran and Singh generalised this result to Hitchin bundles on~$Y$. They considered the category~$\cS^{H}_{Y}$ of strongly semistable Hitchin bundles on $Y$ and defined a~catego\-ry~$\cC^{H}_{Y}$ as the one given by all maps $f\colon \QQ\rightarrow \cS^{H}_{Y}$. They proved that $\cC^{H}_{Y}$ is a neutral Tannakian category and therefore it defines a group scheme~$\cG^{H}_{Y}$, which they called the \emph{Hitchin holonomy group scheme} of $Y$, such that there is an equivalence of categories between $\cC^{H}_{Y}$ and $\operatorname{Rep}\big(\cG^{H}_{Y}\big)$. Unfortunately, this equivalence of categories does not induce a~non-abelian Hodge correspondence for nodal curves as proved in \cite[Theorem~7.2]{bhosle-parameswaran-singh} (see Section~\ref{sec:rephiggs}). Nevertheless, in \cite[Section~5]{bhosle-biswas-hurtubise} Bhosle, Biswas and Hurtubise provide evidence of a~non-abelian Hodge correspondence for generalised Hitchin bundles, although the correspondence is not yet proven.

As we mentioned above, the moduli spaces of Higgs and parabolic Higgs bundles on $X$ are rich in geometric structure. One of their features being that they carry a symplectic structure and a fibration, known as the \emph{Hitchin fibration}, that makes them algebraic completely integrable systems. Bhosle studied an analogue of the Hitchin fibration for generalised parabolic $L$-twisted Hitchin pairs, showed that it is a proper morphism and studied its fibers. Moreover, she found a relation between the moduli of generalised parabolic $L$-twisted Hitchin bundles and the compactified Jacobian of~$Y$.

This review article is based on lectures for a mini-course held at the University of Illinois at Chicago in 2016. We do not intend to be exhaustive. We shall review briefly the results on generalised parabolic bundles and generalised parabolic Hitchin pairs explained above, with the aim to provide the reader with an intuitive view on the subject as well as point to some open problems. For that purpose, we shall not reproduce the proofs for most of the results addressed but refer to the corresponding literature.

\section{Generalised parabolic bundles}\label{sec:1}

Let $X$ be an irreducible non-singular projective curve over an algebraically closed field $k$. Let~$D$ be an effective Cartier divisor on $X$ and $E$ a vector bundle on $X$ of rank $r$ and degree $d$.

A \emph{quasi-parabolic structure} on $E$ over the divisor $D$ consists of a flag of vector subspaces of the vector space $E|_{D}=E\otimes \cO_{D}$, i.e.,
\begin{gather}\label{def:flag}
\cF(E)\colon \ E|_{D}=F_{0}(E)\supset F_{1}(E)\supset F_{2}(E)\supset \cdots \supset F_{m}(E)=0.
\end{gather}

Let $\{D_{j}\}_{j=1}^{n}$ be a set of finitely many disjoint effective Cartier divisors on~$X$. A \emph{quasi-parabolic bundle} $E$ (QPB in short), is a vector bundle $E$ together with quasi-parabolic structures on each~$D_{j}$. We denote it as $(E,\cF(E))$ where $\cF(E)=\big(\cF(E)^{1}, \ldots, \cF(E)^{n}\big)$ where each $\cF(E)^{j}$ is a flag as in~(\ref{def:flag}).

An \emph{homomorphism of QPBs} $f\colon (E,\cF(E))\rightarrow (E',\cF(E'))$ on $X$, is a homomorphism of bundles $f\colon E\rightarrow E'$ which maps the flag $\cF(E)^{j}$ to the flag $\cF(E')^{j}$ for all $j$.

A \emph{parabolic structure} on $E$ over an effective divisor $D$ consists of a quasi-parabolic structure on $E$ over $D$ and a vector of real numbers $\alpha=(\alpha_{1},\ldots,\alpha_{ m})$ such that
\begin{gather}\label{def:weights}
 0\le \alpha_{1}< \cdots< \alpha_{m}<1.
 \end{gather}
These $\alpha$ are called \emph{weights} associated to the flag.

A \emph{generalised parabolic bundle} (from now on shortened as GPB) is a vector bundle $E$ together with parabolic structures over finitely many disjoint divisors $\{D_{j}\}_{j=1}^{n}$. We denote it by a~triple $(E,\cF(E),\alpha(E))$ where $\alpha(E)=\big(\alpha(E)^{1},\ldots, \alpha(E)^{n}\big)$ is the collection of vectors of weights corresponding to each divisor, i.e., $\alpha(E)^{j}=\big(\alpha^{j}_{1},\ldots, \alpha^{j}_{m_{j}}\big)$ for each~$j$. Note that we dropped $E$ from the notation for convenience.

\begin{Remark} We can recover the usual notion of parabolic bundle, as defined in \cite{mehta-seshadri:1980}, by considering generalised parabolic bundles where $D_{j}$ consist of a single point in~$X$ for all~$j$. In this sense, generalised parabolic bundles do generalise the usual notion of parabolic bundles.
\end{Remark}

Let $\ell^{j}_{i}=\dim F^{j}_{i-1}(E)/F^{j}_{i}(E)$ for $i=1, \ldots, m_{j}$. We define
\begin{gather*}
w_{D_{j}}(E)=\sum_{i=1}^{m_{j}}\ell^{j}_{i}\alpha^{j}_{i}, \qquad \textrm{and}\qquad w(E)=\sum_{j=1}^{n}w_{D_{j}}(E).
\end{gather*}

A \emph{homomorphism of GPBs} $f\colon (E, \cF(E),\alpha(E))\rightarrow (E', \cF(E'),\alpha(E'))$ on $X$ with disjoint set of divisors $\{D_{j}\}_{j=1}^{n}$ is a homomorphism of the underlying bundles $f\colon E\rightarrow E'$ such that
\begin{gather*}
f|_{D_{j}}(F^{j}_{i}(E))\subset F^{j}_{k}(E'), \qquad \textrm{whenever}\qquad \alpha^{j}_{k-1}(E')<\alpha^{j}_{i}(E)\le \alpha^{j}_{k}(E').
\end{gather*}

Note that, this definition particularises to the so called non-strongly parabolic homomorphisms in the case of parabolic bundles as in~\cite{seshadri:1982}.

Let $N\subset E$ be a subbundle and $(E,\cF(E),\alpha(E))$ a GPB on $X$ over $\{D_{j}\}_{j=1}^{n}$ then $N$ inherits the parabolic structure from $E$ in the following way.

$(N,\cF(N),\alpha(N))$ is a GPB where
\begin{gather*}
\cF(N)\colon \ F^{j}_{0}(N):=N|_{D_{j}}\supset F^{j}_{1}(N):=F^{j}_{1}(E)\cap N|_{D_{j}}\supset \cdots\supset F^{j}_{r}(N)=0
\end{gather*} and if $\beta_{k}$ is the weight associated to $F_{k}(N)$ then $\beta_{k}: =\alpha_{i}$, where $F^{j}_{i}(E)$ is the smallest subspace such that $F^{j}_{i}(E)\supset F^{j}_{k}(N)$.

By a \emph{subbdundle} of a GPB we mean a subbundle with the inherited parabolic structure.

Finally, we define the \emph{parabolic degree} and the \emph{parabolic slope} of $E$ as
\begin{gather}\label{def:pdeg}
\pdeg(E)=\deg(E)+w(E),\qquad \pmu(E)=\frac{\pdeg(E)}{\rk(E)}.
\end{gather}

A GPB $(E,\cF(E),\alpha(E))$ is \emph{semistable} (respectively \emph{stable}) if for every proper subbundle $N$
\begin{gather*}
\pmu(N)\le\,\mathrm{(resp.}<\mathrm{)} \pmu(E).
\end{gather*}

Once we have defined (semi)stability for generalised parabolic bundles we can study its moduli space.

\begin{Theorem}[{\cite[Theorem 1]{bhosle:1996}}] Let $X$ be an irreducible non-singular projective curve of genus $g$ $(g\ge 0)$ defined over an algebraically closed field. Let $\{D_{j}\}_{j=1}^{n}$ be finitely many disjoint effective divisors on $X$. The moduli space $M(r,d, \cF,\alpha)$ of semistable GPBs on $X$ of fixed rank~$r$, degree~$d$, flags of length~$k$ where~$k$ is independent of~$j$ $($i.e., $\cF(E)^{j}\colon F^{j}_{0}(E)=E|_{D_{j}}\supset F^{j}_{1}(E)\supset \cdots\supset F^{j}_{k}(E)=0$ for all $j=1, \ldots, n)$, and weights fixed and independent of~$j$ is a normal projective variety of dimension
\begin{gather*}
r^{2}(g-1)+1+\sum_{j}\dim\cF_{j},
\end{gather*}
where $\cF_{j}$ is the flag variety of flags of type $\cF(E)^{j}$. Moreover, the subset corresponding to stable GPBs is a non-singular open subvariety.
\end{Theorem}

There are special cases in which the stability conditions of QPBs and GPBs relate. We shall fix some conditions in order to find out this relation.

A QPB $(E,\cF(E))$ it is said to have \emph{flags of length $1$} if
\begin{gather*}\cF(E)^{j}\colon \ E|_{D_{j}}=F^{j}_{0}(E)\supset F^{j}_{1}(E)\supset 0\qquad \mathrm{for\; all}\; j.\end{gather*}

Let $\alpha$ be a real number $0\le \alpha\le 1$. A QPB $(E,\cF(E))$ with flags of length $1$ is called $\alpha$-\emph{semistable} (respectively $\alpha$-\emph{stable}) if for any proper subbundle $N$ of $E$ with the induced ge\-ne\-ralised quasi parabolic structure, one has
\begin{gather*}
\frac{\deg(N)+\alpha \sum\limits_{j=1}^{n}\dim F^{j}_{1}(N)}{\rk(N)}\le \mathrm{(resp. <)} \frac{\deg(E)+\alpha \sum\limits_{i=1}^{n}\dim F^{j}_{1}(E)}{\rk(E)}.
\end{gather*}

\begin{Remark}\label{rmk:stab}Note that for $0\le \alpha<1$ $\alpha$-semistability (or $\alpha$-stability) for a QPB $(E,\cF(E))$ with flags of length $1$, is the same as semistability (or stability) for the GPB $(E,\cF(E),\alpha(E))$ with $\alpha^{j}(E)=(0,\alpha)$ for all $j$.
\end{Remark}

The following proposition sheds light on the necessary conditions for a moduli space of $\alpha$-(semi)stable QPBs (or, by Remark \ref{rmk:stab} for (semi)stable GPBs) to be smooth.

\begin{Proposition}[{\cite[Proposition 3.3]{bhosle:1996}}]\label{prop:stab}
Let $(E,\cF(E))$ be a QPB of rank $r$ with flags of length~$1$ and let $a=\sum\dim F^{j}_{1}(E)$. Then
\begin{enumerate}\itemsep=0pt
\item[$1.$] Suppose that $1- 1/(a(r-1))<\alpha<1$. Then if $(E,\cF(E))$ is $\alpha$-semistable, it is also $1$-semistable. If it is $1$-stable then it is also $\alpha$-stable.
\item[$2.$] Suppose that $(r,d)=1$ and $a$ is an integral multiple of~$r$. Then $(E,\cF(E))$ is $1$-stable if and only if it is $1$-semistable.
\item[$3.$] If all conditions above are satisfied then $\alpha$-stability is equivalent to $\alpha$-semistability and the moduli space $M(r,d,\cF,\alpha)$ of GPBs is non-singular.
\end{enumerate}
\end{Proposition}

Therefore, we get to the following theorem.

\begin{Theorem}[{\cite[Theorem 2]{bhosle:1996}}] Let $M(r,d,\cF,\alpha)$ denote the moduli space of stable GPBs of rank $r$ and degree $d$ with flags of length $1$ on an irreducible non-singular curve $X$ with genus $g$. Denote $a_{j}=\dim F^{j}_{1}(E)$ for all $j$, and $a=\sum\limits_{j=1}^{n}a_{j}$.

If $\alpha(E)=\big(\alpha^{1}(E),\ldots,\alpha^{n}(E)\big)$ is such that $\alpha^{j}(E)=(0,\alpha)$, for all~$j$, $\alpha$ satisfies $1-1/(a(r-1))<\alpha<1$, $(r,d)=1$, and $a$ is an integral multiple of $r$, then $M(r,d,\cF, \alpha)$ is a~fine moduli space.
\end{Theorem}

In particular, in the case of line bundles on $X$ and considering $D$ a reduced effective divisor of degree 2, it is possible to get an explicit description of the moduli space.

\begin{Proposition}[{\cite[Proposition 2.2]{bhosle:1992}}]\label{prop:lineGPB} Let $D=x+z$. The moduli space $M(1,d,\cF,\alpha)$ of GPBs of rank $1$ and degree $d$ with flags of length~$1$ on~$X$ is a $\PP^{1}$-bundle over $J^{d}(X)$, the Jacobian of line bundles of degree~$d$ on~$X$.
\end{Proposition}

\subsection{Fixed determinant GPBs}

Let $\Lambda^{r}$ be the $r$-exterior product of vector spaces, then for any two vector spaces, $E_{1}$ and $E_{2}$, of the same dimension $r$, the direct sum of top exterior products, $\Lambda^{r}(E_{1})\oplus \Lambda^{r}(E_{2})$, is a direct summand of the $r$-exterior product of their direct sum, $\Lambda^{r}(E_{1}\oplus E_{2})$.

There is a canonical projection map
\begin{gather*}
q\colon \ \Lambda^{r}(E_{1}\oplus E_{2})\rightarrow \Lambda^{r}(E_{1})\oplus \Lambda^{r}(E_{2}),
\end{gather*}
such that for a subspace $V\subset E_{1}\oplus E_{2}$ one gets that $q(V)\subset \Lambda^{r}E_{1}\oplus \Lambda^{r} E_{2}$. The operator $\Lambda^{r}$ and the map $q$ easily extend to vector bundles.

Assume that $D_{j}=x_{j}+z_{j}$ for all $j$. We define the top exterior product $\Lambda^{r}(E,\cF(E))$ of a QPB $(E,\cF(E))$ with flags of length~$1$, to be a rank $1$ QPB $(\Lambda^{r}(E), \cF(\Lambda^{r}(E))$ where $\cF(\Lambda^{r}(E))^{j}\colon F^{j}_{0}(\Lambda^{r}(E))\supset q\big(F^{j}_{1}(E)\big)\supset 0$.

For a GPB $(E,\cF(E),\alpha(E))$ we define
\begin{gather*}
\Lambda^{r}(E,\cF(E),\alpha(E)):=\big(\Lambda^{r}(E),\cF(\Lambda^{r}(E)),\alpha(E)\big).
\end{gather*}

Let $p^{j}_{1}\colon F^{j}_{1}(E)\rightarrow E_{x_{j}}$ and $p^{j}_{2}\colon F^{j}_{1}(E)\rightarrow E_{z_{j}}$ be the natural projections. It is possible to globalize the previous construction to the level of moduli spaces, if an only if at least one $p^{j}_{i}$ for each $j$ is an isomorphism. We call $M'$ the subspace of the moduli space $M(r,d,\cF,\alpha)$ satisfying this condition. The determinant map, provided by the extension of the top exterior product to bundles, is well defined on $M'$ (see \cite[Proposition~3.7]{bhosle:1996}). Therefore, we denote by $M_{L}(r,d,\cF,\alpha)$ the closure in $M(r,d,\cF,\alpha)$ of the space of GPBs with rank $r$ degree $d$ and flags of length $1$ such that $\Lambda^{r}(E,\cF(E),\alpha(E))=(L, \cF(L),\alpha(L))$ which we refer to as \emph{the moduli space of stable GPBs of rank $r$ degree $d$ flags of length~$1$ and fixed determinant~$L$}. These considerations lead Bhosle to the following result.

\begin{Theorem}[{\cite[Theorem~3]{bhosle:1996}}]\label{thm:fixdet} Let $(L, \cF(L),\alpha(L))$ be a GPB of rank $1$ and degree $d$ on $X$ with a~set of disjoint divisors $\{D_{j}\}_{j=1}^{n}$. Assume that the projections $p^{j}_{1}\colon F^{j}_{1}(L)\rightarrow L_{x_{j}}$ and $p^{j}_{2}\colon F^{j}_{2}(L)\rightarrow L_{z_{j}}$ are both different from zero for all $j$. Then the moduli space $M_{L}(r,d,\cF,\alpha)$ of GPBs of rank~$r$ degree~$d$ flags of length~$1$ and fixed determinant, i.e., $\Lambda^{r}(E,\cF(E),\alpha(E))=(L, \cF(L),\alpha(L))$, is a normal variety.

Moreover if $(r,d)=1$ and $1-1/(rn(r-1))<\alpha<1$, then $M_{L}(r,d,\cF,\alpha)$ is non-singular.
\end{Theorem}

\section{GPBs and torsion free sheaves on nodal curves}\label{sec:GPBnodal}
The notion of a generalised parabolic bundle has been broadly applied to the theory of torsion free sheaves on singular curves. Bhosle applied these structures in \cite{bhosle:1992,bhosle:1996} to the study of torsion free sheaves on a curve with nodes and cusps as singularities. Just for convenience, along this article, we shall focus only in the case of nodes.

Let $Y$ be an integral (i.e., irreducible and reduced) projective curve over an algebraically closed field with only nodal singularities. From now on we call it \emph{nodal curve}. Let $y_{j}\in Y$, $j=1, \ldots, n$, be the nodal points on $Y$ and
\begin{gather*}
p\colon \ X\longrightarrow Y
\end{gather*}
the normalisation map.

The appropriate QPBs to consider in this situation are $(E,\cF(E))$ of rank $r$ degree $d$ and generalised parabolic structures on the disjoint divisors given by the inverse images $ p^{-1}(y_{j})=x_{j}+z_{j}$, with flags of length~$1$, i.e., for each $j$ we have
\begin{itemize}\itemsep=0pt
\item a divisor $D_{j}=x_{j}+z_{j}$,
\item and a flag $\cF(E)^{j}\colon E_{x_{j}}\oplus E_{z_{j}}=F^{j}_{0}(E)\supset F^{j}_{1}(E)\supset 0$.
\end{itemize}

To each $(E,\cF(E))$ we associate a torsion free sheaf $V$ on $Y$ of rank $r$ and degree $d$ defined by
\begin{gather}\label{eq:fV}
0\longrightarrow V\longrightarrow p_{\ast} (E)\longrightarrow \bigoplus_{j=1}^{n} p_{\ast}(E)\otimes k(y_{j})/F^{j}_{1}(E)\longrightarrow 0,
\end{gather}
where \looseness=1 $k(y_{j})$ is the residue field of the local ring $\cO_{y_{j}}$ at each node (the reader may want to see \cite{bhosle-logares-newstead, cook,seshadri:1982} for further details on torsion free sheaves on nodal curves, QPBs and GPBs). Moreover, the following proposition give us conditions on $(E,\cF(E))$ to produce a vector bundle~$V$ on~$Y$.

\begin{Proposition}[{\cite[Proposition 4.3]{bhosle:1992}}]\label{prop:lf} Let $p_{x_{j}}\colon F^{j}_{1}(E)\rightarrow E_{x_{j}}$ and $p_{z_{j}}\colon F^{j}_{1}(E)\rightarrow E_{z_{j}}$ be the canonical projections. Then
\begin{enumerate}\itemsep=0pt
\item[$1)$] if $p_{x_{j}}$ and $p_{z_{j}}$ are isomorphisms, then $V$ is locally free,
\item[$2)$] if only one projection is an isomorphism and the other one has rank $k$, then
\begin{gather*}V_{y_{j}}\cong k\cO_{y_{j}}\oplus (r-k)\frak{m}_{y_{j}},\end{gather*}
 where $\frak{m}_{j}$ is the maximal ideal of the local ring $\cO_{y_{j}}$, and
 \item[$3)$] if $F_{1}^{j}(E)=M_{x_{j}}\oplus M_{z_{j}}$ such that $M_{x_{j}}\subset k(x_{j})$ and $M_{z_{j}}\subset k(z_{j})$ then $V_{y_{j}}\cong r \frak{m}_{y_{j}}$.
\end{enumerate}
\end{Proposition}

Moreover, stability and semistability of $(E,\cF(E))$ and $V$ relate in the following way.

\begin{Proposition}[{\cite[Proposition 4.2]{bhosle:1992}}]Let $V$ be a torsion free sheaf associated to a~QPB $(E,\cF(E))$ with flags of length~$1$ for all~$j$, then $V$ is $($semi$)$stable if and only if $(E,\cF(E))$ is $1$-$($semi$)$stable.
\end{Proposition}

Fix a vector $k=(k_{1},\ldots, k_{n})$, where $0\le k_{j}\le r$ for all $j=1,\ldots, n$, and define the spaces
\begin{gather*}
M_{k}=\{(E,\cF(E))\,|\,a_{j}+b_{j}=r-k_{j}\},
\end{gather*}
 where $a_{j}$ and $b_{j}$ are the dimensions of the kernels of the projections $p_{x_{j}}$ and $p_{z_{j}}$ defined in Proposition~\ref{prop:lf}. Also, let
 \begin{gather*}
 U_{k}=\{V\,|\,V_{y_{j}}\cong k_{j}\cO_{y_{j}}\oplus (r-k_{j})\frak{m}_{j},\, j=1,\ldots,r\},
 \end{gather*}
 where $\cO_{y_{j}}$ is the local ring at the node $y_{j}$ and $\frak{m}_{j}$ its maximal ideal.

Let $M$ be the set of isomorphism classes of QPBs $(E,\cF(E))$ with rank $r$ degree $d$, and $U$ be the set of isomorphism classes of torsion-free sheaves of rank $r$ and degree $d$ on $Y$. We get stratifications
\begin{gather}\label{eq:strat}
M=\sqcup_{k}M_{k} \qquad \mathrm{and} \qquad U=\sqcup_{k}U_{k}.
\end{gather}

By \cite[Proposition~4.7]{bhosle:1996}, the map
\begin{gather*}
f\colon \ M_{k}\longrightarrow U_{k}
\end{gather*}
defined as $f(E,\cF(E))=V$, as in (\ref{eq:fV}), is such that when restricted to $M_{(r,\ldots, r)}$ is a bijection onto $U_{(r,\ldots, r)}$. Moreover, $f$ sends (semi) stable objects to (semi) stable objects.

The following theorem was first proved for one node in \cite[Theorem~3]{bhosle:1992} and in more generality later in~\cite{bhosle:1996}.

\begin{Theorem}[{\cite[Theorem 4]{bhosle:1996}}]\label{thm:desin} Let $Y$ be a nodal curve and let $p\colon X\rightarrow Y$ be its normalisation. Let~$U(r,d)$ be the moduli space of semistable torsion free sheaves on $Y$ of rank $r$ and degree $d$ and let $M(r,d, \cF,\alpha)$ be the moduli space of semistable GPBs on~$X$, with divisors $D_{j}=p^{-1}(y_{j})=x_{j}+z_{j}$ for all $j$ and flags of length $1$ such that $\dim F^{j}_{1}(E)=r$. Fix $\alpha^{j}=(0,\alpha)$ such that $1-1/(rn(r-1))<\alpha<1$. Then,
\begin{enumerate}\itemsep=0pt
\item[$1)$] $f$ induces a surjective morphism on the moduli spaces
\begin{gather*}f\colon \ M(r,d,\cF,\alpha)\longrightarrow U(r,d),
\end{gather*}
\item[$2)$] when restricted to the stable locus $f\colon M^{s}_{(r,\ldots, r)}\longrightarrow U^{s}_{(r,\ldots,r)}$ is birational,
\item[$3)$] if $(r,d)=1$, then $M(r,d,\cF,\alpha)$ is a desingularisation of $U(r,d)$.
\end{enumerate}
\end{Theorem}

\begin{Remark}\label{rmk:line}In fact, Theorem \ref{thm:desin} implies that, when the rank is one, the moduli space $M(1,d,\cF,\alpha)$ in Proposition~\ref{prop:lineGPB} is isomorphic to the canonical desingularisation of $\overline{J}(Y)$, the compactified Jacobian of line bundles of degree $d$ on $Y$ (see \cite[Theorem~2]{bhosle:1992}).
\end{Remark}

The following theorem gives us an application of the moduli space of GPBs to the desingularisation of the moduli space of torsion free sheaves with fixed determinant.

\begin{Theorem}[{\cite[Theorem 5]{bhosle:1996}}]Let $Y$ be a nodal curve and~$L$ a fixed line bundle on~$Y$. Let $U_{(r,\ldots,r)}$ be as in \eqref{eq:strat} and $U^{L}_{(r,\ldots,r)}\subset U_{(r,\ldots,r)}$ be the closed subset corresponding to bundles with fixed determinant~$L$. Let $U(r,L)$ be the closure of $U^{L}_{(r,\ldots,r)}$ in~$U(r,d)$, the moduli space of semistable torsion free sheaves on~$Y$ or rank~$r$ degree $d=\deg(L)$. Let $M_{\overline{L}}(r,d,\cF,\alpha) $ be the moduli space of semistable GPBs of rank $r$ degree $d$ flags of length $1$ and fixed determinant $\overline{L}=f^{-1}(L)$ provided with the appropriate flags and weights as in Theorem~{\rm \ref{thm:fixdet}}. Then~$f$ induces a birational surjective morphism $M_{\overline{L}}(r,d,\cF,\alpha) \rightarrow U(r,L)$ and when $(r,d)=1$ $M_{\overline{L}}(r,d,\cF,\alpha)$ is a desingularisation of~$U(r,L)$.
\end{Theorem}

\section{GPBs and representations of the fundamental group}\label{sec:repr}

Let $X$ be a compact Riemann surface. Given $\rho\colon \pi_{1}(X)\rightarrow \GL(r,\CC)$, a representation of the fundamental group of $X$ into the general linear group, one may construct a vector bundle on $X$ associated to~$\rho$, $E_{\rho}$, as
\begin{gather*}
E_{\rho}=\big(\widetilde{X}\times\CC^{r}\big)/\pi_{1}(X)\longrightarrow X,
\end{gather*}
where $\widetilde{X}$ denotes the universal cover of $X$ and $\pi_{1}(X)$ acts on $\CC^{r}$ via the representation~$\rho$. Note that the bundle $E_{\rho}$ is a flat bundle by construction, thus $\deg(E_{\rho})=0$.

Weil proved in \cite{weil:1938} that a vector bundle~$E$ on~$X$ comes from a representation~$\rho$ if and only if~$E$ is a direct sum of indecomposable bundles of degree~$0$. Moreover, Narasimhan and Seshadri, in their celebrated work~\cite{narasimhan-seshadri:1965}, proved that a holomorphic vector bundle over an irreducible non-singular curve of genus $g\ge 2$, is polystable of degree~$0$ if and only if it comes from a unitary representation.

A curve $Y$ with $n$ nodes may be seen, homotopically, as obtained from its normalisation~$X$ by attaching a handle to each pair of points, $x_{j}$, $z_{j}$ in $D_{j}\subset X$, given by the inverse image of each node $y_{j}$ in~$Y$ under the normalisation map. Hence, the fundamental group of~$Y$ satisfies
\begin{gather*}
\pi_{1}(Y)\cong \pi_{1}(X)\ast \ZZ\ast \cdots \ast\ZZ,
\end{gather*}
where $\ast$ denotes the $n$-fold free product $\ZZ$ (see \cite[Result~1.7]{bhosle:1995}).

Bhosle proved the following relation between GPBs and representations.

\begin{Theorem}[{\cite[Theorem 1]{bhosle:1995}}]\label{thm:repGPB} Let $(E,\cF(E))$ be a QPB of rank $r$ degree $0$ and length $1$ flag such that $\dim F_{1}(E)=r$. We also assume that all projections $p_{x_{j}}\colon F^{j}_{1}(E)\rightarrow E_{x_{j}}$ and $p_{z_{j}}\colon F^{j}_{1}(E)\rightarrow E_{z_{j}}$ are isomorphisms for all $j$. Then $(E,\cF(E))$ is associated to a representation $\rho\colon \pi_{1}(Y)\rightarrow \GL(r,\CC)$ if and only if $E$ is a direct sum of indecomposable bundles of degree zero. Moreover, if the restriction of $\rho$ to $\pi_{1}(X)$ is unitary, then $E$ is polystable.
\end{Theorem}
\begin{proof}We sketch here the proof in~\cite{bhosle:1995} for one node, in order to give some intuition to the reader on the technical aspects. The proof goes analogously for many nodes.

Let
\begin{gather*}
\rho\colon \pi_{1}(Y)\cong \pi_{1}(X)\ast \ZZ \rightarrow \GL(r,\CC).
\end{gather*}
It factors through a representation $\rho_{X}\colon \pi_{1}(X)\rightarrow \GL(r,\CC)$. By Weil's theorem \cite[Section~7]{weil:1938}, $\rho_{X}$ gives us an indecomposable vector bundle of degree~$0$ on~$X$ which we denote by~$E_{\rho_{X}}$. Take $g:=\rho(1)\in \GL(r,\CC)$ the image of the generator of $\ZZ$, it gives an isomorphism of vector spaces $\sigma\colon (E_{\rho_{X}})_{x} \longrightarrow (E_{\rho_{X}})_{z}$; $v \mapsto g\cdot v$.

We then define $F_{1}(E_{\rho_{X}})$ to be the graph of $\sigma$ in $(E_{\rho_{X}})_{x}\oplus (E_{\rho_{X}})_{z}$.

Conversely, given a QPB $(E,\cF(E))$ indecomposable of degree $0$ we use Weil's theorem to produce a representation $\rho_{X}\colon \pi_{1}(X)\rightarrow \GL(r,\CC)$ which we extend to a representation of $\pi_{1}(Y)$ by assigning to the generator $1\in\ZZ$ the following composition
\begin{gather*}
\rho(1)=p_{z}\circ p_{x}^{-1},
\end{gather*}
where $p_{x}\colon F_{1}(E)\rightarrow E_{x}$, $p_{z}\colon F_{1}(E)\rightarrow E_{z}$ and, since we assumed all projections to be isomorphisms, $\rho(1)$ is an element in $\GL(r,\CC)$.

The second part of the statement is a consequence of Narasimhan--Seshadri's theorem \cite[Theorem~2(A)]{narasimhan-seshadri:1965}.
\end{proof}

Unfortunately, and unlike the parabolic case, it does not lead to an analogue of Narasimhan--Seshadri theorem for generalised parabolic bundles. In fact, Bhosle proved the existence of counterexamples.
\begin{Proposition}[{\cite[Proposition~3.2]{bhosle:1995}}]
Let $Y$ be a nodal curve, with nodes $\{y_{j}\}_{j=1}^{n}$ and geometric genus $g>0$. For $k>1$ or $k=1$ and at least two nodal points, there exist stable vector bundles of rank $2k+1$ and degree $0$ on $Y$ which are not associated to any representation of $\pi_{1}(Y)$ in $\GL(2k+1,\CC)$.
\end{Proposition}

Nevertheless, in \cite[Section 5.2]{bhosle-biswas-hurtubise} the authors introduced a Hitchin--Kobayashi correspondence between the moduli space of GPBs and a moduli space of flat connections on $X$ framed on the points $\{x_{j},z_{j}\}_{j=1}^{n}$. Moreover, they also indicated how the Narasimhan--Seshadri correspondence should be for nodal curves.

\section{Generalised parabolic Hitchin pairs}
Let $X$ be an irreducible projective non-singular curve together with a finite set of disjoint divi\-sors~$\{D_{j}\}_{j=1}^{n}$ as above. Let $L_{0}$ be a line bundle on~$X$.

An \emph{$L_{0}$-twisted generalised parabolic Hitchin pair} on $X$ ($L_{0}$-twisted GPH, for short) of rank~$r$, degree $d$ and parabolic structure over $\{D_{j}\}_{j=1}^{n}$ is a triple $(E,\cF(E), \phi)$ consisting of a QPB $(E,\cF(E))$ together with a bundle homomorphism
\begin{gather*}
\phi\colon \ E\longrightarrow E\otimes L_{0}.
\end{gather*}

Notice that no compatibility condition is assumed for $\phi$ with respect to the generalised parabolic structure.

From now on we only consider flags of length~$1$, and we say that $(E,\cF(E),\phi)$ is a \emph{good} $L_{0}$-twisted GPH if
\begin{gather}\label{eq:good}
\phi_{D_{j}}\big(F^{j}_{1}(E)\big)\subset F^{j}_{1}(E).
\end{gather}
We call $\phi$ the \emph{generalised parabolic Higgs field}.

A \emph{homomorphism} of GPHs
\begin{gather*}
f\colon \ (E_{1},\cF(E_{1}), \phi_{1}) \longrightarrow (E_{2},\cF(E_{2}), \phi_{2})
\end{gather*}
is a homomorphism of vector bundles $f\colon E_{1}\rightarrow E_{2}$ which is compatible with their respective generalised parabolic Higgs fields, $\phi_{1}$ and $\phi_{2}$, and preserves the generalised parabolic structures for all $j$, i.e., $f_{D_{j}} \big(F^{j}_{1}(E_{1})\big)\subset F^{j}_{1}(E_{2})$.

For each subbundle $(N,\cF(N))$ of $(E,\cF(E))$ we say that it is $\phi$-\emph{invariant} if $\phi(N)\subset N\otimes L_{0}$. In such a case, $(N,\cF(N),\phi|_{N})$ becomes an $L_{0}$-twisted GPH.

Let $\alpha$ be a real number, such that $0<\alpha\le 1$. An $L_{0}$-twisted GPH $(E,\cF(E),\phi)$ is \emph{$\alpha$-semistable} (respectively $\alpha$-\emph{stable}) if for every proper, $\phi$-invariant, quasi-parabolic subbundle $(N, \cF(N))$, one has
\begin{gather*}
\frac{\deg(N)+\alpha\sum\limits_{j=1}^{n} \dim F^{j}_{1}(N)}{\rk(N)} \le \mathrm{(resp. <)}\frac{\deg(E)+\alpha \sum\limits_{j=1}^{n} \dim F^{j}_{1}(E)}{\rk(E)}.
\end{gather*}

Recall that for flags of lenght $1$ an $\alpha$-(semi)stable QPB $(E,\cF(E))$ is equivalent to a (semi)\-stab\-le GPB $(E,\cF(E),\alpha(E))$ with weights $\alpha(E)^{j}=(0,\alpha)$ for all $j$.

Bhosle constructed in~\cite{bhosle:2014} the moduli space of $L_{0}$-twisted GPHs.

\begin{Theorem}[{\cite[Theorem 4.8]{bhosle:2014}}]Let $X$ be a non-singular projective curve of genus $g$. Fix a line bundle $L_{0}$ on $X$ and fix $\alpha$, $0<\alpha\le 1$. Let $\{D_{j}\}_{j=1}^{n}$ be a finite set of disjoint divisors of $X$. Then there exists a moduli scheme $\cM(r,d,L_{0},\cF,\alpha)$ of $\alpha$ semistable $L_{0}$-twisted GPHs of rank~$r$, degree~$d$, flags of length~$1$ and weights $\alpha(E)^{j}=(0,\alpha)$.

Moreover, those which are good GPHs form a closed subscheme,
\begin{gather*}
\cM^{\rm good}(r,d,L_{0},\cF,\alpha)\subset \cM(r,d,L_{0},\cF,\alpha).
\end{gather*}
\end{Theorem}

\section{Hitchin pairs on a nodal curve and GPHs}

Let $Y$ be a nodal curve and $L$ a line bundle on~$Y$. A \emph{$(L$-twisted$)$ Hitchin pair} on~$Y$ is a pair $(V,\varphi)$ consisting of a coherent torsion free sheaf~$V$ on~$Y$ together with a morphism of sheaves
 \begin{gather*}
 \varphi\colon \ V\longrightarrow V\otimes L
 \end{gather*}
 on $Y$.

 There are several choices we may want to do at this point regarding whether we may want~$V$ to be a bundle or~$L$ to be the canonical line bundle on~$Y$, each of these choices lead us to different names for these pairs.

An $L$-twisted Hitchin pair for which $V$ is locally free and $\varphi$ is a bundle homomorphism is called an \emph{$L$-twisted Hitchin bundle}. If~$L$ is the canonical bundle on~$Y$ then a $L$-twisted Hitchin pair $(V,\varphi)$ is called \emph{Higgs pair}. Finally, a Higgs pair for which $V$ is locally free and $\varphi$ a bundle homomorphism is called a \emph{Higgs bundle} on~$Y$.

Fix $L_{0}=p^{\ast}L$. We want to explore the relation between good $L_{0}$-twisted GPHs on $X$ and $L$-twisted Hitchin pairs on~$Y$. This is done in~\cite{bhosle:2014} for general singularities, but since we are focusing only on the case of nodal singularities we shall assume that $D_{j}=p^{-1}(y_{j})=x_{j}+z_{j}$ for all~$j$, and flags have length~$1$ as in Section~\ref{sec:GPBnodal}.

\begin{Proposition}[{\cite[Proposition 2.8]{bhosle:2014}}]\quad
\begin{enumerate}\itemsep=0pt
\item[$1.$] A good $p^{\ast}L$-twisted GPH, $(E,\cF(E), \phi)$, of rank $r$ degree $d$ on $X$ defines an $L$-twisted Hitchin pair $(V,\varphi)$ of rank $r$ and degree~$d$ on~$Y$.
\item[$2.$] If $(V,\varphi)$ is an $L$-twisted Hitchin bundle then $(V,\varphi)$ determines a good $p^{\ast}L$-twisted GPH $(E,\cF(E),\phi)$ where $E=p^{\ast}(V)$ and
\begin{gather*}
\phi=p^{\ast}\varphi\colon \ E\longrightarrow E\otimes p^{\ast}L.
\end{gather*}
\end{enumerate}
\end{Proposition}

\begin{proof}1.~The pair $(E,\cF(E))$ on $X$ determines a torsion free sheaf on $Y$ defined as in (\ref{eq:fV}), that is, $V$ is the kernel in the following exact sequence
\begin{gather}\label{seq:tfY}
0\longrightarrow V\longrightarrow p_{\ast}(E)\longrightarrow p_{\ast}\left(\bigoplus_{j}\frac{E\otimes \cO_{D_{j}}}{F^{j}_{1}(E)}\right)\longrightarrow 0.
\end{gather}
Recall that, because of the choices made on the generalised parabolic structure, $\deg(E)=\deg(V)$ and $\rk(E)=\rk(V)$. Tensoring (\ref{seq:tfY}) by $L$ and recalling (\ref{eq:good}) the induced map on the kernels give us $\varphi$, i.e.,
\begin{gather*}
\xymatrix{0 \ar[r] & V \ar[r]\ar[d]_{\varphi} & p_{\ast}(E)\ar[r]\ar[d]^{p_{\ast}\phi} & \bigoplus_{j}\frac{p_{\ast}(E\otimes \cO_{D_{j}})}{p_{\ast}(F^{j}_{1}(E))}\ar[r] \ar[d]^{\bigoplus_{j} (p_{\ast}\phi)_{y_{j}}}& 0\\
0 \ar[r] & V \ar[r]& p_{\ast}(E)\ar[r]\ & \bigoplus_{j}\frac{p_{\ast}(E\otimes \cO_{D_{j}})}{p_{\ast}(F^{j}_{1}(E))}\ar[r] & 0.
}
\end{gather*}

2.~An $L$-twisted Hitchin bundle $(V,\varphi)$ determines a good $p^{\ast}L$-twisted GPH in the following way. By Theorem~\ref{thm:desin}, $V$ determines a QPB $(E,\cF(E))$ with flags of length~$1$ where $E=p^{\ast}V$. We define $\phi=p^{\ast}\varphi$. Notice that, $p_{\ast}(F^{j}_{1}(E))$ should be such that $V_{y_{j}}\subset (p_{\ast}(E))_{y_{j}}$. Since $\varphi_{y_{j}}(V_{y_{j}})\subset V_{y_{j}}$, $\phi$ satisfies condition~(\ref{eq:good}).
\end{proof}

There is an equivalence of categories given by the following theorem.

\begin{Theorem}[{\cite[Theorem 2.9]{bhosle:2014}}]\label{thm:functor}There exists a functor
\begin{gather*}
F\colon \ \GPH^{\rm good}(r, d, p^{\ast}L) \longrightarrow \Hitchin(r,d,L)
\end{gather*}
between the category of good $p^{\ast}L$-twisted GPHs of rank $r$ and degree $d$ on $X$ and the category of $L$-twisted Hitchin pairs of rank $r$ degree $d$ on $Y$.

Moreover, the restriction of $F$ to the full subcategory $\GPH^{{\rm good},b}(r,d,p^{\ast}L)$ of good $p^{\ast}L$-twisted $\GPH$ on~$X$ that correspond to $L$-twisted Hitchin bundles on~$Y$ gives an equivalence of categories.
\end{Theorem}

\subsection{The moduli spaces}

The functor $F$ in Theorem \ref{thm:functor} preserves semistability for $\alpha=1$, i.e., $(E,F(E),\phi)$ is $1$-semistable if and only if the associated Hitchin pair $(V,\varphi)$ is semistable. But the stability is only induced in one direction, that is, if the Hitchin pair $(V,\varphi)$ is stable then its corresponding GPH $(E,F(E),\phi)$ is $1$-stable (see Bhosle \cite[Theorem~2.9]{bhosle:2014}).

We denote by $\cH(r,d,L)$ the moduli space of semistable $L$-twisted Hitchin pairs on $Y$ of rank~$r$ and degree $d$, and by $\cH^{b}(r,d,L)$ the moduli space of semistable $L$-twisted Hitchin bundles. The construction of this moduli space is sketched in~\cite{bhosle-parameswaran-singh} following Simpson's construction in~\cite{simpson:1994, simpson:1994-2}.

\begin{Theorem}[{\cite[Theorem 4.9]{bhosle:2014}}]\label{thm:birat} Let $Y$ be a nodal curve with nodes $y_{j}\in Y$ $j=1, \ldots, n$, and let $p\colon X\rightarrow L$ be its normalisation. Fix $D_{j}=p^{-1}(y_{j})=x_{j}+z_{j}$ for all~$j$. Then, if $\alpha=1$ or $\alpha<1$ and close to $1$, there exist a birational morphism
\begin{gather*}
f\colon \ \cM^{\rm good}(r,d,L_{0},\cF,\alpha)\longrightarrow \cH(r,d,L)
\end{gather*}
from the moduli space of semistable good $L_{0}=p^{\ast}L$-twisted GPH of rank $r$ and degree $d$ with flags of length $1$ and $\alpha(E)^{j}=(0,\alpha)$ for all $j$, to the moduli of semistable $L$-twisted Hitchin pairs of rank $r$ and degree $d$.

When $\alpha=1$, the image $f\big(\cM^{\rm good}(r,d,L_{0},\cF,\alpha)\big)$ contains $\cH^{b}(r,d,L)$, i.e., all semistable Hitchin bundles, and when $\alpha<1$, but close to $1$\footnote{The condition \emph{close to $1$} for $\alpha$ is explicitly determined in~\cite{bhosle:2014}.}, $f\big(\cM^{\rm good}(r,d,L_{0},\cF,\alpha)\big)$ contains all stable Hitchin bundles.
\end{Theorem}

We denote by $\cH^{b,s}(r,d,L)$ the moduli space of stable $L$-twisted Hitchin bundles of rank $r$ and degree $d$.

\subsection{The Hitchin map}

Let $\GPH(r, d, L_{0})$ be the set of $L_{0}$-twisted GPH of rank $r$ and degree~$d$ on~$X$, and $\GPH^{{\rm good},b}(r,$ $d, L_{0})$ be the set of good $L_{0}$-twisted GPH which determine Hitchin bundles. Let $\Hitchin (r, d, L)$ be the set of $L$-twisted Hitchin pairs on~$Y$, and $\Hitchin^{b}(r, d, L)$ the set of Hitchin bundles.

There are maps
\begin{gather*}
h_{X}\colon \ \GPH(r,d,L_{0}) \longrightarrow \bigoplus_{i=1}^{r} H^{0}\big(X, L^{i}_{0}\big),\qquad
(E, F(E), \phi) \mapsto \bigoplus_{i=1}^{r}\tr \phi^{i}
\end{gather*}
and
\begin{gather*}
h_{Y}\colon \ \Hitchin^{b}(r,d,L) \longrightarrow \bigoplus _{i=1}^{r} H^{0}\big(Y,L^{i}\big),\qquad (V, \varphi)\mapsto \bigoplus_{i=1}^{r}\tr \varphi^{i},
\end{gather*}
where $\phi^{i}=\phi\circ \cdots\circ \phi$, and $\varphi^{i}=\varphi\circ\cdots\circ\varphi$, $i$-times, and $\tr$ is the trace operator.

Notice that $H^{0}\big(Y,L^{i}\big)\subset H^{0}\big(L^{i}\otimes p_{\ast}(\cO_{X})\big)$ and $h_{X}\big(\GPH^{\rm good}(r,d,L_{0})\big)\subset \bigoplus _{i=1}^{r} H^{0}\big(Y,L^{i}\big)$.

Hence, $h_{X}$ and $h_{Y}$ fit into the following diagram
\begin{gather*}
\xymatrix{
h_{X}\colon \ \GPH^{{\rm good},b} (r,d,L_0)\ar[r]\ar@<1ex>[d]^{f} & \bigoplus _{i=1}^{r} H^{0}\big(Y,L^{i}\big)\ar@{=}[d]\\
h_{Y}\colon \ \Hitchin^{b}(r,d,L)\ar[r] \ar@<1ex>[u]^{p^{\ast}}& \bigoplus _{i=1}^{r} H^{0}\big(Y,L^{i}\big).
}
\end{gather*}

The maps $p^{\ast}$ and $f$ induce homeomorphisms on the corresponding moduli spaces.

At the level of moduli spaces there is a map
\begin{gather*}
h_{\cM}\colon \ \cM(r,d,L_{0},\cF,\alpha)\longrightarrow A:=\bigoplus_{i=1}^{r} H^{0}\big(X,S^{i}L_{0}\big),
\end{gather*}
where $S^{i}L_{0}$ denotes the symmetric product of the line bundle $L_{0}$ (see \cite[Section~5]{bhosle:2014} for details).

\begin{Theorem}[{\cite[Corollary 5.2]{bhosle:2014}}]\quad
\begin{enumerate}\itemsep=0pt
\item[$1.$] The restriction of $h_{\cM}$ to $\cM^{\rm good}(r,d,L_{0},\cF,\alpha)$ is a pro\-per morphism.
\item[$2.$] The Hitchin map $h_{\cM^{\rm good}}$ defines a proper morphism
\begin{gather*}
h_{\cH}\colon \ f\big(\cM^{\rm good}(r,d,L_{0},\cF,\alpha)\big)\rightarrow A
\end{gather*}
\item[$3.$] Let $A':=\bigoplus_{i=1}^{r}H^{0}\big(Y,L^{i}\big)\subset A$. Then for $\alpha=1$, there is a commutative diagram
\begin{gather*}
\xymatrix{
\cM^{\rm good}(r,d,L_{0},\cF,\alpha) \ar[r]^{h_{\cM^{\rm good}}}\ar[d]^{f}& A'\\
\cH^{b}(r,d,L)\subset f\big(\cM^{\rm good}(r,d,L_{0},\cF,\alpha)\big)\subset \cH(r,d,L)\ar[ru]_{h_{\cH^{b}}} &
}
\end{gather*}
Hence $h_{\cM^{\rm good}}= h_{\cH^{b}} \circ f$ on $f^{-1}\big(\cH^{b}(r,d,L)\big)$ and $h_{\cM^{\rm good}}\circ p^{\ast}= h_{\cH^{b}}$ on $\cH^{b}(r,d,L)$.

When $\alpha<1$ but close to $1$, as only $\cH^{b,s}(r,d,L)\subset f\big(\cM^{\rm good}(r,d,L_{0},\cF,\alpha)\big)$ we need to consider the maps defined on the intersection
\begin{gather*}
\cH^{b}(r,d,L)\cap f\big(\cM^{\rm good}(r,d,L_{0},\cF,\alpha)\big).
\end{gather*}
\end{enumerate}
\end{Theorem}

\subsection{Example in rank one}

In the following we review briefly the description in \cite[Section 2.4]{bhosle:2014}.

Consider $(V,\varphi)\in \cH(1,d,L)$ then $V$ can be locally free or not. If $V$ is locally free then $H^{0}(Y,\End(V)\otimes L)\cong H^{0}(Y,L)$. Therefore, the space parametrising the space of possible Higgs fields for the pair $(V,\varphi)$ is $H^{0}(Y,L)$.

If $V$ is torsion free but not locally free we get that the space parametrising Higgs fields for~$V$ is
\begin{gather*}
H^{0}(Y,\End(V)\otimes L)\cong H^{0}(Y,L\otimes p_{\ast}\cO_{X})\cong H^{0}(X,p^{\ast}L).
\end{gather*}

Hence, the moduli space $\cH(1,d,L)$ of $L$-twisted Hitchin pairs of degree $d$ on $Y$ is isomorphic to
\begin{gather*}
\big[J(Y)\times H^{0}(Y,L)\big] \sqcup \big[\big(\overline{J}(Y)-J(Y)\big)\times H^{0}(Y,L\otimes p_{\ast}(\cO_{X}))\big],
\end{gather*}
where $J(Y)$ denotes the Jacobian of $Y$ and $\overline{J}(Y)$ its compactification.

Notice that $\cH(1,d, L)\subset \overline{J}(Y)\times H^{0}(L\otimes p_{\ast}\cO_{X})$.

Recall from Remark \ref{rmk:line} that $M(1,d,\cF,\alpha)$ is isomorphic to the canonical desingularisation of $\overline{J}(Y)$ then
\begin{gather}\label{eq:M1}
\cM(1,d,L_{0},\cF,\alpha)\cong M(1,d,\cF,\alpha)\times H^{0}(Y,L\otimes p_{\ast}(\cO_{X})).
\end{gather}

Moreover, the canonical desingularisation (see \cite[Proposition 12.1]{oda-seshadri})
\begin{gather*}\nu\colon \ \widetilde{J}(Y)\rightarrow \overline{J}(Y)\end{gather*} induces a morphism
\begin{gather*}
\nu\times i\colon \ \widetilde{J}(Y)\times H^{0}(X,L_{0})\longrightarrow \overline{J}(Y)\times H^{0}(Y,L\otimes p_{\ast}(\cO_{X})).
\end{gather*}

\begin{Lemma}[{\cite[Lemma 2.12]{bhosle:2014}}]
\begin{gather*}
\cM^{\rm good}(1,d,L_{0})\cong \big[\nu^{-1}(J(Y))\times H^{0}(Y,L)\big]\sqcup \big[\nu^{-1}\big(\overline{J}(Y)-J(Y)\big)\times H^{0}(X,L_{0})\big]
\end{gather*}
with $L$ such that $L_{0}=p^{\ast}L$ and $H^{0}(Y,L)\subset H^{0}(X,p^{\ast}L)$ by the inclusion given by the pullback $\varphi\mapsto p^{\ast}\varphi$. Moreover, the morphism of moduli spaces
\begin{gather*}
f\colon \ \cM^{\rm good}(1,d,L_{0})\longrightarrow \cH(1,d,L)
\end{gather*}
restricts to an isomorphism on $\nu^{-1}(J(Y))\times H^{0}(Y,L)$ and it is a surjective two-to-one map on $\nu^{-1}\big(\overline{J}(Y)-J(Y)\big)\times H^{0}(X,L_{0})$.
\end{Lemma}

The Hitchin map in this case particularises in the following way.

For the moduli space of $L_{0}$-twisted GPH we get
\begin{gather*}
h_{\cM}\colon \ \cM(1,d,L_{0},\cF,\alpha)\longrightarrow H^{0}(X,L_{0}), \qquad (E,\cF(E),\phi)\mapsto \phi,
\end{gather*}
which has fibres isomorphic to $\widetilde{J}(Y)\cong M(1,d,\cF,\alpha)$ by (\ref{eq:M1}).

The map $h_{\cM}$ induces a map on the restriction to $\cM^{\rm good}(1,d,L_{0}\cF,\alpha)$
\begin{gather*}
h_{\cM^{\rm good}}\colon \ \cM^{\rm good}(1,d,L_{0},\cF,\alpha)\longrightarrow H^{0}(X,L_{0}),
\end{gather*}
which descends, through the birational map $f$ from Theorem \ref{thm:birat}, to the Hitchin map
\begin{gather*}
h_{\cH}\colon \ \cH(1,d,L)\longrightarrow H^{0}(Y,L\otimes p_{\ast}(\cO_{X})), \qquad (V,\varphi)\mapsto \varphi.
\end{gather*}

The following lemma describes the corresponding fibres.
\begin{Lemma}[{\cite[Lemma 2.13]{bhosle:2014}}]Let $s\in H^{0}(X,p^{\ast}L)\cong H^{0}(Y,L\otimes p_{\ast}(\cO_{X}))$ then
\begin{gather*}
h_{\cM^{\rm good}}^{-1}(s)= \begin{cases} \widetilde{J}(Y) & \mathrm{if} \ s\in H^{0}(Y,L),\\
\Pic ^{d}(X)\sqcup \Pic^{d}(X) & \mathrm{if} \ s\in H^{0}(Y,L\otimes p_{\ast}(\cO_{X}))- H^{0}(Y,L)\end{cases}
\end{gather*}
and
\begin{gather*}
h_{\cH}^{-1}(s)= \begin{cases} \overline{J}(Y) & \mathrm{if} \ s\in H^{0}(Y,L),\\
\overline{J}(Y)-J(Y) & \mathrm{if} \ s\in H^{0}(Y,L\otimes p_{\ast}(\cO_{X}))- H^{0}(Y,L).\end{cases}
\end{gather*}
\end{Lemma}

\section[Representations of the fundamental group and Higgs bundles on the nodal curve]{Representations of the fundamental group\\ and Higgs bundles on the nodal curve}\label{sec:rephiggs}
In Section~\ref{sec:repr} we pointed out that there is no analogue of the celebrated Narasimhan--Seshadri theorem in the context of generalised parabolic structures. Nevertheless, Theorem~\ref{thm:repGPB} proved a~relation between unitary representations of a~nodal curve $Y$ and generalised parabolic bundles on its normalisation $X$. We recall here the work in~\cite{bhosle-parameswaran-singh} describing the situation for representations of the fundamental group of $Y$ into $\GL(r,\CC)$.

Let $Y$ be a nodal curve with $n$ nodes and $p\colon X\rightarrow Y$ its normalisation. Let $\rho\colon \pi_{1}(Y)\longrightarrow \GL(r,\CC)$ be a representation. Recall from Section~\ref{sec:repr} that $\pi_{1}(Y)\cong \pi_{1}(X)\ast \ZZ\ast \cdots\ast\ZZ$, where $\ast$ denotes the free product and it is taken $n$ times.

Let $\rho_{X}$ be the restriction of $\rho$ to $\pi_{1}(X)$. By the non-abelian Hodge theorem \cite[Proposition~1.5]{simpson:1992} $\rho_{X}$ defines a Higgs bundle $(E_{\rho_{X}},\phi_{\rho_{X}})$ on $X$ of rank $r$ and degree~$0$. Moreover, the element $g_{j}=\rho(1)\in\GL(r,\CC)$ gives an isomorphism $\sigma_{j}\colon (E_{\rho_{X}})_{x_{j}}\rightarrow (E_{\rho_{X}})_{z_{j}}$ on the fibres over~$x_{j}$ and~$z_{j}$. This provides us a QPB $(E,\cF(E))$ on~$X$ and consequently, a bundle $E_{\rho}$ on $Y$ by identifying the fibres on~$x_{j}$ and~$z_{j}$ through~$\sigma_{j}$.

The pushforward of the Higgs field $\phi_{\rho_{X}}$ satisfies
\begin{gather*}
p_{\ast}\phi_{\rho_{X}}\colon \ p_{\ast}E_{\rho_{X}}\longrightarrow p_{\ast}(E_{\rho_{X}}\otimes K_{X})=p_{\ast}(p^{\ast}(E_{\rho})\otimes K_{X})=E_{\rho}\otimes p_{\ast}K_{X},
\end{gather*}
where $K_{X}$ is the canonical bundle on $X$. Note that $p_{\ast}(K_{X})\subset K_{Y}$ and $p^{\ast}K_{Y}=K_{X}\big(\sum_{j}(x_{j}+z_{j})\big)$, so we can define the Higgs field $\phi_{\rho}\colon E_{\rho}\rightarrow E_{\rho}\otimes K_{Y}$ as the composition
\begin{gather*}
\phi_{\rho}\colon \ E_{\rho}\hookrightarrow p_{\ast}E_{\rho_{X}}\stackrel{p_{\ast}\phi_{\rho_{X}}}{\longrightarrow} E_{\rho}\otimes p_{\ast}K_{X}\hookrightarrow E_{\rho}\otimes K_{Y}.
\end{gather*}

By construction one has that $\rk(E_{\rho})=\rk(E_{\rho_{X}})=r$ and $\deg(E_{\rho})=\deg(E_{\rho_{X}})=0$.

The (semi)stability of $(E_{\rho},\phi_{\rho})$ on $Y$ corresponds to the (semi)stability of $(E_{\rho_{X}},\phi_{\rho_{X}})$ \cite[Lemma~7.1]{bhosle-parameswaran-singh}. Furthermore, a Higgs bundle $(E,\phi)$ on $Y$ is called \emph{strongly semistable} if for every $k\ge 1$ $\big({\otimes}^{k}E,\otimes^{k} \phi\big)$ is semistable, so the following theorem holds.

\begin{Theorem}[{\cite[Theorem 7.2]{bhosle-parameswaran-singh}}]\label{thm:Higgsrepr}
Let $Y$ be a nodal curve, then to a representation $\pi_{1}(Y)\rightarrow \GL(r,\CC)$ one can associate a strongly semistable Higgs bundle of rank $r$ and degree $0$. If the restriction of $\rho$ to the fundamental group of the normalisation $X$ of $Y$ is irreducible with genus $X\ge 2$ then the associate Higgs bundle is stable.
\end{Theorem}

The converse of Theorem~\ref{thm:Higgsrepr} is not true in general. In \cite[Section~7.3]{bhosle-parameswaran-singh} the authors constructed a stable Higgs bundle of degree $0$ on $Y$ which does not come from a representation of~$\pi_{1}(Y)$.

\section{Further comments and open problems}

The non-abelian Hodge correspondence actually refers to real analytic isomorphisms between three moduli spaces, which originally were: the moduli space of representations of the fundamental group of a compact Riemann surface $X$, the moduli spaces of flat connections of a principal bundle on $X$, and the moduli space of Higgs bundles on $X$. It can be thought of as a combination of the Riemann--Hilbert correspondence relating moduli spaces of flat connections and representations of the fundamental group into a Lie group $G$ and the Hitchin--Kobayashi correspondence relating Hermitian--Einstein connections and holomorphic bundle structures. The latter actually provides a general set up for the Narasimhan--Seshadri theorem, which can be thought as a Hitchin--Kobayashi correspondence for $G=U(r)$.

In \cite{bhosle-biswas-hurtubise} Bhosle, Biswas and Hurtubise, motivated by the question of how Narasimhan--Seshadri theorem should look like for generalised parabolic bundles, constructed compact moduli spaces of Grassmanian-framed bundles over a Riemann surface and compact moduli spaces for the representation theoretic side of the correspondence, that is moduli spaces of unitary connections with simple poles and arbitrary residues. Their construction shows that GPBs should then correspond to unitary connections with simple poles on~$X$ such that the eigenvalues at the points~$x_j$ are minus the eigenvalues at the points~$z_j$. All of this is developed to some extent in \cite[Section~5.2]{bhosle-biswas-hurtubise} but not yet proven.

The study of Higgs bundles on singular curves has attracted the attention of many authors in the last years and there are many advances in this direction using different techniques. In~\cite{balaji-barik-nagaraj} Balaji, Barak and Nagaraj construct a degeneration of the moduli space of Higgs bundles on smooth curves, as the smooth curve degenerates to an irreducible nodal curve~$Y$ with a single node. In~\cite{ mazzeo-swoboda-weiss-witt, swoboda:2017} the authors also study degenerations of a curve and the corresponding moduli spaces but from an analytic point of view. Nevertheless, it is still an open question to understand the relation of these moduli spaces coming from degenerations in the algebraic or analytic context with the moduli spaces considered by Bhosle in~\cite{bhosle:2014}.

Finally, in \cite{bhosle:2001} Bhosle extends the notion of generalised parabolic bundles to principal $G$-bundles for $G$ any reductive algebraic group. The extension to generalised parabolic principal $G$-Higgs bundles is still an open problem as well as how the Hitchin fibration looks like and the spectral theory involved for any~$G$.

\subsection*{Acknowledgements}

The author would like to thank U.N.~Bhosle for very useful explanations, the referees for a~detailed review of the manuscript, and L.~Schaposnik for the organisation of the \emph{Workshop on the geometry and physics of Higgs bundles~I}. The author also acknowledges support by L.P.~Schaposnik's UIC Start up fund NSF RTG grant DMS-1246844, the GEAR Network NSF DMS grants 1107452, 1107263 and 1107367 ``RNMS: GEometric structures And Representation varieties'', and the Marie Sklodowska Curie grant GREAT~- DLV-654490.

\pdfbookmark[1]{References}{ref}
\LastPageEnding

\end{document}